\documentclass[11pt,a4paper]{amsart}

\usepackage{amsthm}

\usepackage[utf8]{inputenc}
\usepackage{bbm,amssymb,mathtools,mathrsfs}

\usepackage[usenames,dvipsnames]{xcolor}

\usepackage[colorlinks,linkcolor=BrickRed,citecolor=OliveGreen,urlcolor=black,hypertexnames=true]{hyperref}

\usepackage[margin=3.5cm]{geometry}

\usepackage{enumerate}

\usepackage{bbding}
\usepackage{calligra}
\usepackage[T1]{fontenc}

\allowdisplaybreaks

\makeatletter
\def\l@subsection{\@tocline{2}{-4pt}{2.5pc}{5pc}{}}
\renewcommand\tocchapter[3]{%
  \indentlabel{\@ifnotempty{#2}{\ignorespaces#2.\quad}}#3%
}
\newcommand\@dotsep{4.5}
\def\@tocline#1#2#3#4#5#6#7{\relax
  \ifnum #1>\c@tocdepth 
  \else
    \par \addpenalty\@secpenalty\addvspace{#2}%
    \begingroup \hyphenpenalty\@M
    \@ifempty{#4}{%
      \@tempdima\csname r@tocindent\number#1\endcsname\relax
    }{%
      \@tempdima#4\relax
    }%
    \parindent\z@ \leftskip#3\relax \advance\leftskip\@tempdima\relax
    \rightskip\@pnumwidth plus1em \parfillskip-\@pnumwidth
    #5\leavevmode\hskip-\@tempdima{#6}\nobreak
    \leaders\hbox{$\m@th\mkern \@dotsep mu\hbox{.}\mkern \@dotsep mu$}\hfill
    \nobreak
    \hbox to\@pnumwidth{\@tocpagenum{#7}}\par
    \nobreak
    \endgroup
  \fi}
\makeatother
\AtBeginDocument{%
\makeatletter
\expandafter\renewcommand\csname r@tocindent0\endcsname{0pt}
\makeatother
}
\def\l@subsection{\@tocline{2}{0pt}{2.5pc}{5pc}{}}

\input xy
\xyoption{all}

\DeclareMathOperator{\Var}{Var}

\DeclareMathOperator{\id}{id}
\DeclareMathOperator{\Sym}{Sym}
\DeclareMathOperator{\Skew}{Skew}

\DeclareMathOperator{\tr}{tr}
\DeclareMathOperator{\sgn}{sgn}
\DeclareMathOperator{\Tr}{Tr}
\DeclareMathOperator{\Trd}{Trd}

\DeclareMathOperator{\End}{End}
\DeclareMathOperator{\Ker}{Ker}

\DeclareMathOperator{\diag}{diag}
\DeclareMathOperator{\sign}{sign}

\DeclareMathOperator{\Nil}{Nil}

\DeclareMathOperator{\Int}{Int}
\DeclareMathOperator{\rk}{rank}

\DeclareMathOperator*{\bigperp}{\raisebox{-.8ex}{\scalebox{2}{$\perp$}}}

\DeclareMathOperator*{\midperp}{\raisebox{-.7ex}{\scalebox{1.5}{$\perp$}}}

\DeclareMathOperator{\Supp}{Supp}

\newcommand{\N}{\mathbb{N}}
\newcommand{\Z}{\mathbb{Z}}

\newcommand{\CM}{\mathscr{M}}
\newcommand{\CN}{\mathscr{N}}

\newcommand{\Sper}{\operatorname{Sper}}
\newcommand{\Spec}{\operatorname{Spec}}

\newcommand{\fp}{\mathfrak{p}}

\newcommand{\s}{\sigma}

\newcommand{\ox}{\otimes}
\newcommand{\x}{\times}

\newcommand{\Qf}{\mathrm{qf}}

\newcommand{\ve}{\varepsilon}

\newcommand{\ovl}{\overline}

\newcommand{\qf}[1]{\langle #1\rangle}

\newcommand{\sm}{\setminus}

\newcommand{\vf}{\varphi}
\newcommand{\vt}{\vartheta}

\newcommand{\knp}{\bullet}

\newcommand{\io}[1]{\prescript{\iota}{}{#1}}

\newcommand{\ns}{\mathrm{ns}}

\newcommand{\op}{\mathrm{op}}

\newtheorem{thm}{Theorem}[section]

\newtheorem{lemma}[thm]{Lemma}

\newtheorem{cor}[thm]{Corollary}
\newtheorem{prop}[thm]{Proposition}

\theoremstyle{definition}
\newtheorem{defn}[thm]{Definition}
\newtheorem{exa}[thm]{Example}
\newtheorem*{exa*}{Example}

\theoremstyle{remark}
\newtheorem{rem}[thm]{Remark}

\numberwithin{equation}{section}

\definecolor{dgreen}{rgb}{0.1,0.5,0.1}
\definecolor{purple}{rgb}{0.49, 0.06, 0.51}
\definecolor{orange}{rgb}{1, 0.5, 0}

\begin{document}

\title[Continuity of total signature maps]{Continuity of total signature maps for Azumaya algebras with involution}

\author[V. Astier]{Vincent Astier}
\author[T. Unger]{Thomas Unger}
\address{School of Mathematics and Statistics, 
University College Dublin, Belfield, Dublin~4, Ireland}
\email{vincent.astier@ucd.ie}
\email{thomas.unger@ucd.ie}

\subjclass{16H05, 11E39, 13J30, 16W10, 06F25}
\date{\today}
\keywords{Azumaya algebras, involutions, hermitian forms, orderings, 
signatures, 
continuity, real algebra}


\begin{abstract}
  In this paper we continue our investigation of signatures of hermitian forms
  over Azumaya algebras with involution over commutative rings. We show that
  the approach used in an earlier paper for central simple algebras can be
  extended to Azumaya algebras and leads to a natural way of choosing the
  signature of a hermitian form at a given ordering, producing total signatures
  of hermitian forms that are continuous functions on the real spectrum of the
  base ring.
\end{abstract}


\maketitle


\parskip\smallskipamount
\tableofcontents
\parskip0pt


\section{Introduction}

In \cite{A-U-Az-PLG} we showed that signatures of hermitian forms can be
defined for Azumaya algebras with involution over a commutative ring $R$, and
used them to prove a version of Sylvester's inertia theorem and, in case $R$ is
semilocal, Pfister's local-global principle. The definition of the
signature at $\alpha \in \Sper R$ relies on the application of a Morita
equivalence which, when chosen differently, can change the sign of the result.
This is a consequence of a reduction to the central simple case, cf.
\cite{A-U-Kneb}. There does not seem to be a direct canonical way to choose
this Morita equivalence, which presents a problem if we consider the total
signature map of a hermitian form, and want to obtain a continuous function
defined on $\Sper R$. 

In this paper we show that the approach used in \cite{A-U-Kneb} for central
simple algebras with involution can be extended to Azumaya algebras with
involution: The sign of the signature can be chosen with the help of a
``reference form'' (Section~\ref{sec-eta-sign}), and this definition of
signature leads to total signatures of hermitian forms that are continuous
(Section~\ref{sec-cont}).

The results in this paper are very similar to those obtained for central simple
algebras with involution, but in many cases the existing proofs could not
simply be extended to Azumaya algebras. This required the introduction of new
tools, occasionally leading to simpler arguments. One such promising new tool
is obtained in Section~\ref{secstarmult} via a further investigation of a
pairing of hermitian forms, studied by Garrel in the central simple case, cf.
\cite{garrel-2023}.

\section{Preliminaries}

In this paper all rings are assumed unital and associative with $2$ invertible
and all fields are assumed to have characteristic different from $2$.
Let $R$ be a commutative ring. An
$R$-algebra $A$ is an \emph{Azumaya $R$-algebra} if $A$ is a faithful finitely
generated projective $R$-module and the map 
\begin{equation}\label{swm}
  A\ox_R  A^{\op}\to \End_R(A),\ a\ox b^\op\mapsto [x\mapsto axb]
\end{equation}
is an isomorphism of $R$-algebras (here $A^\op$ denotes the \emph{opposite
algebra} of $A$, which coincides with $A$ as an $R$-module, but with twisted
multiplication $a^\op b^\op=(ba)^\op$). The centre $Z(A)$ is equal to $R$. See \cite[III, (5.1)]{knus91}, for example. We recall the following results:

\begin{prop}[{\cite[pp.~11, 12]{Saltman99}}]\label{basic-Az}\mbox{}
  \begin{enumerate}[{\quad\rm (1)}]
    \item If $K$ is a field, then $A$ is an Azumaya $K$-algebra if and only
      if $A$ is a central simple $K$-algebra.
    \item If $A$ is an Azumaya $R$-algebra and $f: R \rightarrow S$ is a
    morphism of rings, then $A \ox_R S$ is an Azumaya $S$-algebra.
  \end{enumerate}
\end{prop}

\begin{defn}[{\cite[Section~1.4]{first23}}]\label{first-def}  
  We say that $(A,\s)$ is an \emph{Azumaya
  algebra with involution over $R$} if the following conditions hold:
  \begin{itemize}
    \item $A$ is an $R$-algebra with $R$-linear involution $\s$;
    \item $A$ is separable projective over $R$;
    \item the homomorphism $R \rightarrow A$, $r \mapsto r \cdot 1_A$ identifies
      $R$ with the set $\Sym(Z(A),\s)$ of $\s$-symmetric elements of $Z(A)$.
  \end{itemize}
\end{defn}

This definition is motivated by:

\begin{prop}[{\cite[Proposition~1.1]{first23}}]\label{prop:Az}
  Let $A$ be an $R$-algebra. Then
  $A$ is Azumaya over $Z(A)$ and $Z(A)$ is finite \'etale over $R$ if and only
  if $A$ is projective as an $R$-module and separable as an $R$-algebra.
\end{prop}

\noindent\textbf{Assumption for the remainder of the paper:} $R$ is a commutative ring
(with $2$ invertible)
and $(A,\s)$ is an Azumaya
$R$-algebra with involution.
\bigskip
 
Note that $A$ is Azumaya over $Z(A)$ by Proposition~\ref{prop:Az}, 
but may not be  Azumaya over $R$.
Indeed, ``Azumaya algebra with involution'' means ``Azumaya 
algebra-with-involution'' rather than ``Azumaya-algebra with involution''.

In the special case of central simple algebras with involution, we recall that
if $(B,\tau)$ is such an algebra, then $\tau$ is said to be of the first kind
if $Z(B) \subseteq \Sym(B,\tau)$, and of the second kind (or of unitary type)
otherwise.  Involutions of the first kind are further divided into those of
orthogonal type and those of symplectic type, depending on the dimension of
$\Sym(B,\tau)$, cf.  \cite[Sections~2.A and 2.B]{BOI}.

A hermitian module over $(A,\s)$ is a pair $(M,h)$ where $M$
is a finitely generated projective right $A$-module and $h\colon M\x M \to A$ is
a hermitian form. We often refer to $(M,h)$ as a hermitian form and do not 
always indicate what $M$ is.
For $a_1,\ldots, a_\ell
\in \Sym(A,\s)$ we denote by $\qf{a_1,\ldots, a_\ell}_\s$
the diagonal  hermitian form 
\[
  A^\ell\x A^\ell \to A,\ (x,y)\mapsto \sum_{i=1}^\ell
  \s(x_i)a_i y_i.
\]
Note that not every hermitian form is isometric to a diagonal hermitian form. 
We denote isometry of forms by $\simeq$.

Let $\Trd_A\colon A\to Z(A)$ denote the reduced trace of $A$, cf.
\cite[IV, \S 2]{KO}, and recall that it is additive and $Z(A)$-linear.
Furthermore, $\Trd_A$ commutes with scalar extensions of $Z(A)$  since 
its computation does not depend on the choice of splitting ring, cf. 
\cite[IV, Proposition~2.1]{KO}. In fact, $\Trd_A$ also commutes with scalar 
extensions of $R$ (the case of interest to us):

\begin{lemma}[{\cite[Lemma~4.1]{A-U-Az-PLG}}]\label{tr-extend} 
  Let $R'$ be a commutative ring that contains $R$. Then for all $a\in A$,
  \[
    \Trd_A(a) \ox_R 1_{R'} = \Trd_{A \ox_R R'}(a \ox_R 1_{R'}).
  \]
\end{lemma}

We define the trace quadratic form of $A$ as follows:
\[
  T_A: A \to Z(A),\ x\mapsto \Trd_A(x^2).
\]
The associated  symmetric bilinear form is $A\x A \to Z(A), (x,y)\mapsto \Trd_A(xy)$. We note that $T_A$ is nonsingular (the proof uses the same
arguments as the proof of \cite[Lemma~4.3]{A-U-Az-PLG}).
\medskip

We denote the real spectrum of $R$ by $\Sper R$ and by
\[
 \mathring H(r) := 
 \{\alpha \in \Sper R \mid r > 0 \text{ at } \alpha\} 
    = \{\alpha \in \Sper R \mid r \in \alpha \setminus -\alpha\}
\]
(for $r \in R$) the sets that form the 
standard subbasis of open sets of the Harrison topology on $\Sper R$.
We recall that there is a second, finer, topology on $\Sper R$, the constructible
topology, which has the finite boolean combinations of sets of the form
$\mathring H(r)$ (for $r \in R$) as a basis of open sets, cf. \cite[Definition~3.4.3]{KSU}. 
The space $\Sper R$ is compact (by which we
mean quasi-compact) for both topologies, cf. \cite[Theorem~3.4.5 and
Corollary~3.4.6]{KSU}.

Let $\alpha \in\Sper R$.
The support of $\alpha$ is the prime ideal $\Supp(\alpha) := \alpha \cap
-\alpha \in \Spec R$, and   
we denote
\begin{itemize}
  \item by $\rho_\alpha$ the canonical map
    $R \rightarrow \kappa(\alpha):=\Qf (R/\Supp(\alpha))$,
  \item by $\bar \alpha$ the ordering induced by $\alpha$ on
    the field of fractions $\kappa(\alpha)$,
  \item and by $k(\alpha)$ a real closure of $\kappa(\alpha)$ at $\bar\alpha$.
\end{itemize}

We also define
\[(A(\alpha), \sigma(\alpha)) := (A \ox_R \kappa(\alpha), \s \ox \id)\]
and
\[(A_\alpha, \sigma_\alpha) := (A \ox_R k(\alpha), \s \ox \id).\]
Clearly, $(A_\alpha, \s_\alpha) = (A(\alpha), \s(\alpha)) \ox_{\kappa(\alpha)} k(\alpha)$.

\begin{rem}\mbox{}\label{A-extend}
  \begin{enumerate}[{\quad\rm (1)}]
    \item Recall from \cite[Second paragraph of Section~1.4]{first23} 
    that if $T$ is a
      commutative $R$-algebra, then $(A \ox_R T, \s \ox \id)$ is an Azumaya
      algebra with involution over $T$ with $Z(A \ox_R T) = Z(A) \ox_R T$.
    \item If $R = F$ is a field, then $(A,\s)$ is a central simple $F$-algebra
      with involution in the sense of 
      \cite[Start of Sections~2.A and 2.B]{BOI}, i.e., $A$ is an $F$-algebra, 
      $\dim_F A$ is finite,
      $F$ is identified with $\Sym(A,\s)\cap Z(A)$ and either $A$ is simple 
      (in which case $Z(A)$ is
      a field) or is the product of two simple $F$-algebras 
      (in which case $Z(A) \cong F \x
      F$).
    \item  In particular, we obtain from the
      two previous points that $(A(\alpha), \s(\alpha))$ is a central simple
      $\kappa(\alpha)$-algebra with involution and that $(A_\alpha, \s_\alpha)$
      is a central simple $k(\alpha)$-algebra with involution, both in the sense
      of \cite{BOI}.

    \item Since $k(\alpha)$, $k(\alpha)(\sqrt{-1})$ and $(-1,-1)_{k(\alpha)}$
    are (up to isomorphism) the only finite-dimensional division
    $k(\alpha)$-algebras, it follows from the previous parts of this remark that
    $A_\alpha \cong
    M_{n_\alpha}(D_\alpha)$ for some $n_\alpha \in \N$, where $D_\alpha$ is one
    of
  \[k(\alpha),\ k(\alpha)(\sqrt{-1}),\ (-1,-1)_{k(\alpha)},
    k(\alpha) \x
        k(\alpha), (-1,-1)_{k(\alpha)} \x (-1,-1)_{k(\alpha)},
  \]
    where the final two cases can only occur if $\s_\alpha$ is of the second
    kind.
  \end{enumerate}

\end{rem}

\section{$\CM$-signatures}

We recall the following definition from \cite{A-U-Az-PLG}:

\begin{defn}\label{def:sig_h}
	Let $h$ be a hermitian form over $(A,\s)$ and let $\alpha\in \Sper R$. 
	Then $h\ox \kappa(\alpha)$ is a hermitian form over
	the central simple algebra with involution $(A(\alpha), \s(\alpha))$ 
  and we define
	the \emph{$\CM$-signature} of $h$ at $\alpha$ by
  \[\sign^{\CM}_\alpha h := \sign^{\CM_{\bar\alpha}}_{\bar \alpha} 
  (h \ox_R \kappa(\alpha)),\]
  where $\CM_{\bar\alpha}$ is a Morita equivalence as in 
  \cite[Section~3.2]{A-U-Kneb}; see
  \cite[Definition~3.1]{A-U-Az-PLG} and the discussion following it.
\end{defn}

\begin{rem}\label{signprop}
  The $\CM$-signature $\sign^{\CM}_\alpha$ inherits all the properties from the
  map $\sign^{\CM_{\bar\alpha}}_{\bar \alpha}$. For instance, it is additive
  and $\sign^\CM_\alpha (q\ox h)=\sign_\alpha q \cdot \sign^\CM_\alpha h $ if
  $q$ is quadratic over $R$ and $h$ hermitian over $(A,\s)$, cf.
  \cite[Proposition~3.6]{A-U-Kneb}.
\end{rem}

The main drawback of the $\CM$-signature at $\alpha \in \Sper R$ is that 
there is no canonical choice of Morita equivalence $\CM_{\bar\alpha}$, 
and that an arbitrary change in the sign of the signature of a form
can be obtained by taking a different Morita equivalence (cf. 
\cite[Proposition~3.4 and first paragraph of Section~3.3]{A-U-Kneb}).
This is in particular a problem
if we hope to consider the total signature of a form as a continuous 
function on $\Sper R$.
This problem is solved with the definition of $\eta$-signature in 
Section~\ref{sec-eta-sign}. 

\begin{rem}\label{choice-Morita}
  If there are several instances of $\CM$-signatures of forms over $(A,\s)$ in
  a statement or formula, we use the same choice of Morita equivalence for each
  of them.
\end{rem}

\begin{lemma}\label{sign-bounded}
  Let $(M,h)$ be a hermitian form over $(A,\s)$. There is $k \in \N$ depending only on $A$ and $M$, such that $|\sign^\CM_\alpha h| \le k$ for every $\alpha \in \Sper R$.
\end{lemma}

\begin{proof}
  Let $S$ be a finite set of generators of $M$ over $A$. Since $A$ is
  projective over $R$, the rank function on $A$ is continuous 
  (cf. \cite[p.~12]{Saltman99})
  and thus takes only
  finitely many values $n_1^2, \ldots, n_t^2$ (which are all squares). Let $N =
  \mathrm{lcm}(n_1, \ldots, n_t)$.
  The form $N \x h \ox_R \kappa(\alpha)$ is defined over $(A(\alpha),
  \s(\alpha))$, where $A(\alpha)$ has degree $n_i$ for some $1 \le i \le t$. By
  \cite[Lemma~2.2]{A-U-pos} (and since $n_i$, and thus $N$, are multiples of
  the matrix size of $A(\alpha)$), there exist $r\in\N$ and
  $a_1, \ldots, a_r \in \Sym(A(\alpha)^\x, \s(\alpha)) \cup\{0\}$ 
  such that $N \x h \ox_R \kappa(\alpha) \simeq \qf{a_1, \ldots,
  a_r}_{\s(\alpha)}$. By considering the underlying modules of both forms, it
  follows that $r \le \dim_{\kappa(\alpha)} (M^N \ox_R \kappa(\alpha)) \le N
  \cdot |S|$, and thus
  \begin{align*}
      |\sign^\CM_\alpha h| &\le N \cdot |\sign^\CM_\alpha h|\\
      &= N  \cdot |\sign^{\CM_{\bar \alpha}}_{\bar \alpha} h \ox 
      \kappa(\alpha)| \\
        &= |\sign^{\CM_{\bar \alpha}}_{\bar \alpha} 
        \qf{a_1, \ldots, a_r}_{\s(\alpha)}| \\
        &\le r\cdot n_{\bar \alpha}\\
        &\le N \cdot |S| \cdot n_{\bar \alpha},
    \end{align*}
  where the penultimate inequality comes from \cite[Proposition~4.4(iii)]{A-U-PS}. The
  result follows by the definition of $n_{\bar \alpha}$ (cf. 
  Remark~\ref{A-extend} or
  \cite[p.~352]{A-U-PS} since clearly $n_{\bar \alpha} \le n_i$ for some $i \in
  \{1,\ldots, t\}$).
\end{proof}

\subsection{Special case: quadratic \'etale
extensions of $R$}\label{quad-et}

Let $S$ be a 
quadratic \'etale $R$-algebra. We recall that $S$ has a unique
standard involution $\vt$, cf. \cite[I, (1.3)]{knus91}, and that 
$(S,\vt)$ is
an Azumaya algebra with involution over $R$ by \cite[Lemma~2.5]{A-U-Az-PLG}.

We refer to \cite[Section~2.2]{A-U-Az-PLG} for the following observations:
The trace $\Tr_{S/R}$ satisfies $\Tr_{S/R}(x)=\vt(x)+x$ for all
$x\in S$. Let $h$ be a hermitian form over $(S,\vt)$. Then $\Tr_{S/R}(h)$
is a symmetric bilinear form over $R$ with associated quadratic form
$q_h(x)=\Tr_{S/R}(h(x,x))=2h(x,x)$. If $h$ is nonsingular or hyperbolic, then
$\Tr_{S/R}(h)$ is also nonsingular or hyperbolic, respectively. 
\medskip

Let $\alpha\in \Sper R$. If $S\ox_R \kappa(\alpha)\cong \kappa(\alpha)\x \kappa(\alpha)$, then $\vt\ox\id_R$ is the exchange involution $(x,y)\mapsto (y,x)$, and we let
\begin{equation}\label{signcentre1}
  \sign_\alpha h:= 0
\end{equation}
since (the nonsingular part of) $h\ox_R \kappa(\alpha)$ is hyperbolic, cf.  
\cite[Proof of Lemma~2.1(iv)]{A-U-Kneb}.

Thus we may assume that $S\ox_R \kappa(\alpha)$ is a quadratic field extension 
of $\kappa(\alpha)$, and so  we can use the second case of 
\cite[Equation~(2.2)]{A-U-sign-pc}, noting that $q_h$ and the
quadratic form $b_h$ from \cite[Section~2.1]{A-U-sign-pc} differ by a factor 
$2$, so that they have the same signature.
Since $\Tr_{S/R}$ commutes with scalar extension (cf. 
\cite[Remark~2.9]{A-U-Az-PLG}), we thus define the signature of $h$ at $\alpha$
in this case by
\begin{equation}\label{signcentre2}
  \sign_\alpha h:= \tfrac{1}{2} \sign_\alpha q_h
  =\tfrac{1}{2} \sign_{\bar\alpha} (q_h\ox \kappa(\alpha))
\end{equation}
in view of \cite[Equation~(2.2)]{A-U-sign-pc}.

Note that as a consequence of \cite[Remark~2.5]{A-U-sign-pc}, the definition
of $\sign_\alpha h$ in \eqref{signcentre2} is a special case of 
$\sign^\CM_\alpha h$ for a suitable choice of hermitian Morita equivalence.

\subsection{The set of nil-orderings is clopen}

\begin{defn}\label{def-Nil}
  We
  define the \emph{set of nil-orderings} of $(A,\s)$ to be
  \begin{align*}
    \Nil[A,\s] &:=\{\alpha \in \Sper(R) \mid \sign_\alpha^{\CM} = 0\} \\
        &\phantom{:}= \{\alpha \in \Sper(R) \mid \sign^{\CM_{\bar\alpha}}_{\bar \alpha} = 0\} \text{ by \cite[Lemma~3.7]{A-U-Az-PLG}}.
  \end{align*}
\end{defn}

As observed in \cite[Remark~3.6]{A-U-Az-PLG}, this definition is equivalent
to \cite[Definition~3.7]{A-U-Kneb} in case $(A,\s)$ is a central simple
algebra with involution, which allows us to make the following observation:

\begin{rem}\label{cases-nil}
  We use the notation from Remark~\ref{A-extend}(4).
  For $\alpha \in \Sper R$, we have that $\alpha \in \Nil[A,\s]$ if and only if
\[D_\alpha =
\begin{cases}
  (-1,-1)_{k(\alpha)} & \text{ if $\sigma_\alpha$ is orthogonal,} \\
  k(\alpha) & \text{ if $\s_\alpha$ is symplectic,} \\
  k(\alpha) \x k(\alpha) \text{ or } (-1,-1)_{k(\alpha)} \x (-1,-1)_{k(\alpha)} & \text{ if $\s_\alpha$ is unitary.}
\end{cases}\]
\end{rem}

\begin{lemma}\mbox{}\label{trd}
  Let $F$ be a field (of characteristic not $2$).
    \begin{enumerate}[{\quad\rm (1)}]
    \item For $a,b\in F^\x$  let 
      $(a,b)_F$ denote the quaternion $F$-algebra with generators $i,j$
      such that $i^2=a, j^2=b, ij=-ji=:k$. 
      If $A=(a,b)_F$ with $a,b\in F^\x$, then
    \[
      T_A\simeq \qf{2,2a,2b,-2ab}.
    \]
    \item If $A=M_n(F)$ with $n\in \N$, then
    \[
      T_A\simeq n\x \qf{1}\perp \frac{n(n-1)}{2}\x \qf{-1,1}.
    \]
  \end{enumerate}
\end{lemma}

\begin{proof} 
  The proof of these well-known facts is by straightforward computation, see 
  \cite{Lewis94} for example. 
  
  (1) Let $\alpha = x_0 +x_1i +x_2j +x_3k \in A$. Then $\Trd_A(\alpha)=2x_0$.
  The $F$-basis $\{1,i,j,k\}$ of $A$ is orthogonal for $T_A$ and we get
  $T_A\simeq \qf{2,2a,2b,-2ab}$.
  
  (2) Let $\alpha=\sum_{i,j=1}^n \alpha_{ij}e_{ij}\in A$, where 
  $\{e_{ij}\}_{i,j=1,\ldots,n}$ is the standard matrix basis of $A$. Then
  $\Trd_A(\alpha)=\tr(\alpha)=\sum_{i=1}^n \alpha_{ii}$ and  
  $T_A\simeq n\x  \qf{1}\perp \frac{n(n-1)}{2}\x \qf{-1,1}$, where 
  $n\x  \qf{1}$ comes from $e_{11},\ldots, e_{nn}$, while the pairs 
  $\{e_{ij}, e_{ji}\}_{i\not=j}$ contribute the $\frac{n(n-1)}{2}$ hyperbolic
  planes.
\end{proof}

\begin{lemma}\label{prop:rc}
  Let $F$ be a real closed field, let $H=(-1,-1)_F$ and let $n\in \N$.
  \begin{enumerate}[{\quad\rm (1)}]
    \item If $A=H$, then $\sign T_A=-2$.
    \item If $A=M_n(F)$, then $\sign T_A=n$.
    \item If $A=M_{n}(H)$, then $\sign T_A=-2n$
    \item If $A=M_n(F)\x M_n(F)$, then $\sign T_A=2n$.
    \item If $A=M_{n}(H)\x M_{n}(H)$, 
    then $\sign T_A= -4n$.
  \end{enumerate}
\end{lemma}

\begin{proof}
  This is a straightforward application of Lemma~\ref{trd} and the fact that if
  $A$ and $B$ are central simple $F$-algebras, then
    \[T_{A\ox_F B} \simeq T_A \ox T_B \quad\text{and}\quad T_{A\x B}\simeq T_A 
      \perp T_B,
    \]
  cf. \cite[Lemma~2.1]{Lewis94} for the first equality, while the second
  equality follows from \cite[p.~121]{Reiner2003}.
\end{proof}

We can now present an elementary proof of the 
following result that already appeared as \cite[Corollary~6.5]{A-U-Kneb}:

\begin{prop}\label{csa-nil}
  Let $F$ be a field (of characteristic not $2$) and
  let
  $(B,\tau)$ be a central simple $F$-algebra 
  with involution. Then $\Nil[B,\tau]$ is a clopen subset of $X_F$.  
\end{prop}

\begin{proof}
  We use Lemmas~\ref{tr-extend} and 
  \ref{prop:rc}, and the fact that the total signature map of a
  quadratic form is continuous, cf. \cite[VIII, Proposition~6.6]{LQF2}.
  
  a) Assume that $\s$ is of the first kind and that $\dim_F B=n^2$.
  If $\s$ is symplectic, then
  \[
    \Nil[B,\tau]=\{P\in X_F \mid B\ox_F F_P \cong M_n(F_P)\} = (\sign_\bullet 
    T_B)^{-1}(n).
  \]
  If $\s$ is orthogonal, then
  \[
    \Nil[B,\tau]=\{P\in X_F \mid B\ox_F F_P \cong 
    M_{n/2}\bigl((-1,-1)_{F_P}\bigr)\} = (\sign_\bullet 
    T_B)^{-1}(-n).
  \]
  
  b) Assume that $\s$ is of the second kind and let $K$ denote the centre 
  of $B$. 
  If $K\cong F\x F$, then 
  $\Nil[B,\tau]=X_F$ by \cite[Definition~3.7]{A-U-Kneb}, see the comment after
  Definition~\ref{def-Nil}.
  Assume thus that $K$ is a field and let $\dim_K B=n^2$.
  Consider the  clopen sets $S_1$ and $S_2$, defined as follows:
  \begin{align*}
    S_1&:= \{P\in X_F \mid B\ox_F F_P \cong M_n(F_P)\x M_n(F_P)\} = 
    (\sign_\bullet T_B)^{-1}(2n)\\
    S_2&:= \{P\in X_F \mid B\ox_F F_P \cong M_{n/2}\bigl((-1,-1)_{F_P}\bigr)\x 
    M_{n/2}\bigl((-1,-1)_{F_P}\bigr)\}\\ 
    &\phantom{:}=  (\sign_\bullet T_B)^{-1}(-2n).
  \end{align*}
  The statement follows since $\Nil[B,\tau]=S_1 \dot\cup S_2$.  
\end{proof}

\begin{lemma}\label{clopen-sets}
  The 
  following
  subsets of $\Sper R$ are clopen for the Harrison topology:
  \begin{align*}
    C_1&:= \{\alpha \in \Sper R \mid \s_\alpha \text{ is orthogonal}\},\\
    C_2&:= \{\alpha \in \Sper R \mid \s_\alpha \text{ is symplectic}\},\\
    C_3&:= \{\alpha \in \Sper R \mid \s_\alpha \text{ is unitary}\},\\
    C_4&:= \{\alpha \in \Sper R \mid A \ox_R k(\alpha) \cong 
    M_{n}(k(\alpha))\},\\
    C_5&:= \{\alpha \in \Sper R \mid A \ox_R k(\alpha) \cong 
    M_{n}(k(\alpha)(\sqrt{-1}))\},\\
    C_6&:= \{\alpha \in \Sper R \mid A \ox_R k(\alpha) \cong
     M_{n/2}((-1,-1)_{k(\alpha)})\},\\
    C_7&:= \{\alpha \in \Sper R \mid A \ox_R k(\alpha) \cong
     M_{n}(k(\alpha)) \x M_{n}(k(\alpha))\},\\
    C_8&:= \{\alpha \in \Sper R \mid A \ox_R k(\alpha) \cong
     M_{n/2}((-1,-1)_{k(\alpha)}) \x 
     M_{n/2}((-1,-1)_{k(\alpha)})\}.
  \end{align*}
\end{lemma}

\begin{proof}
  Let $r_\alpha= \dim_{k(\alpha)} A_\alpha $ for every $\alpha\in \Sper R$.

  If $\s_\alpha$ is of the first kind, then $r_\alpha = m_\alpha^2$ for some 
  $m_\alpha 
  \in \N$, and $\dim_{k(\alpha)} \Sym(A_\alpha, \s_\alpha) = 
  m_\alpha(m_\alpha+1)/2$ if
  $\s_\alpha$ is orthogonal, while $\dim_{k(\alpha)} \Sym(A_\alpha, \s_\alpha) = m_\alpha(m_\alpha-1)/2$ if
  $\s_\alpha$ is symplectic, cf. \cite[Proposition~2.6]{BOI}.

  If $\s_\alpha$ is of the second kind, then $r_\alpha=2m_\alpha^2$ for some 
  $m_\alpha \in \N$ (see
  \cite[p.~21, 3 lines before Proposition~2.15]{BOI}) and $\dim_{k(\alpha)}
  \Sym(A_\alpha, \s_\alpha) = m_\alpha^2$ by \cite[Proposition~2.17]{BOI}.

  Therefore, the three possible values for $\dim_{k(\alpha)} \Sym(A_\alpha,
  \s_\alpha)$ are
  \[\sqrt{r_\alpha}(\sqrt{r_\alpha}+1)/2, \ \sqrt{r_\alpha}(\sqrt{r_\alpha}-1)/2 \text{ and } r_\alpha/2.\]
  Since these values are all different and since
  $\dim_{k(\alpha)} \Sym(A_\alpha, \s_\alpha)=\rk_{\alpha^0} \Sym(A, \s)$,
  where $\alpha^0:=\Supp(\alpha)$,
  it follows that
  \begin{align*}
    C_1&= \{\alpha \in \Sper R \mid \rk_{\alpha^0} \Sym(A, \s) = 
          \sqrt{r_\alpha}(\sqrt{r_\alpha}+1)/2\},\\
    C_2&= \{\alpha \in \Sper R \mid \rk_{\alpha^0} \Sym(A,\s) = 
          \sqrt{r_\alpha}(\sqrt{r_\alpha}-1)/2\},\\
    C_3&= \{\alpha \in \Sper R \mid \rk_{\alpha^0} \Sym(A,\s) = r_\alpha/2\}.    
  \end{align*}
  Since $A = \Sym(A,\s) \oplus \Skew(A,\s)$ (note that this uses the
  assumption $2\in R^\x$)
  and is projective over $R$, we get
  that $\Sym(A,\s)$ is projective over $R$. In particular, the map 
  \[
    \rk_\bullet \Sym(A,\s) : \Spec R \rightarrow \Z
  \]
  is continuous, cf. \cite[p.~12]{Saltman99}.  
  Since the map $\Sper R
  \rightarrow \Spec R$, $\alpha \mapsto \alpha^0$ is continuous (with the
  Harrison topology on $\Sper R$, cf. \cite[Proposition~3.3.3]{KSU}), 
  the composition
  of both maps is continuous. It follows that $C_1$, $C_2$ and $C_3$ are clopen.

  We now consider the sets  $C_4$ to $C_8$. Observe that (using
  Proposition~\ref{prop:rc})
  \begin{itemize}
    \item If $A \ox_R k(\alpha) \cong M_{m_\alpha}(k(\alpha))$, then 
    $\rk_{\alpha^0} A =
      \dim_{k(\alpha)} A_\alpha = m_\alpha^2$ and $\sign_\alpha T_A = m_\alpha 
      = \sqrt{r_\alpha}$.
    \item If $A \ox_R k(\alpha) \cong M_{m_\alpha/2}((-1,-1)_{k(\alpha)})$, 
    then 
    $\rk_{\alpha^0}
      A = \dim_{k(\alpha)} A_\alpha = m_\alpha^2$ and $\sign_\alpha T_A = 
      -m_\alpha =
      -\sqrt{r_\alpha}$.
    \item If $A \ox_R k(\alpha) \cong M_{m_\alpha}(k(\alpha)) \x 
    M_{m_\alpha}(k(\alpha))$, then 
    $\rk_{\alpha^0}
      A = \dim_{k(\alpha)} A_\alpha = 2m_\alpha^2 $ and 
      $\sign_\alpha T_A = 2m_\alpha =
      \sqrt{2r_\alpha}$.
    \item If $A \ox_R k(\alpha) \cong M_{m_\alpha/2}((-1,-1)_{k(\alpha)}) \x
      M_{m_\alpha/2}((-1,-1)_{k(\alpha)})$, then $\rk_{\alpha^0} A = 
      \dim_{k(\alpha)} A_\alpha
      = 2m_\alpha^2$ and $\sign_\alpha T_A = -2m_\alpha = -\sqrt{2r_\alpha}$.
  \end{itemize}
  Therefore,
  \[C_4 =
  \{\alpha \in \Sper R \mid \sign_\alpha T_A = \sqrt{r_\alpha}\},\]
  which is clopen in $\Sper R$ for the Harrison topology since $T_A$ is a
  nonsingular
  quadratic form over $R$, cf. \cite[Section~1.3]{mahe82}. Similar arguments show that the 
  sets $C_6$, $C_7$ and $C_8$  are clopen. 
  Finally, $C_5$ is clopen since $C_5= \Sper R \sm (C_4\cup C_6\cup C_7\cup 
  C_8)= C_3 \sm (C_7\cup C_8)$. 
\end{proof}

\begin{thm}\label{nil-clopen}
  The set $\Nil[A,\s]$
  is a clopen subset of $\Sper R$ for the Harrison topology
  (and thus also for the constructible topology).
\end{thm}

\begin{proof}   
  By Lemma~\ref{clopen-sets} and the
  description of nil-orderings in Remark~\ref{cases-nil}, we
  have
  \begin{equation}\label{eq-nil}
    \Nil[A,\s]= (C_1\cap C_6)\cup (C_2\cap C_4)\cup (C_3\cap C_7)\cup 
    (C_3\cap C_8),
  \end{equation}
  which proves the result. 
\end{proof}

\subsection{The $\CM$-signature is almost 
$*$-multiplicative}\label{secstarmult}

We recall from \cite[Section~4.4]{A-U-Az-PLG} 
that there exists a pairing $*$ of hermitian forms 
$(M_1,h_1)$ and $(M_2,h_2)$
over $(A,\s)$ 
(first studied in detail by N.~Garrel in 
\cite{garrel-2023} for central simple algebras with involution)
such
that $h_1 * h_2$ is a hermitian form over $(Z(A),\iota)$, where 
$\iota:=\s|_{Z(A)}$ (and thus is a quadratic form over $R$ if $Z(A)=R$), and which preserves orthogonal sums, isometries and
nonsingularity, cf. \cite[Corollary~4.8]{A-U-Az-PLG}. 
The form $h_1 * h_2$ is defined on the $Z(A)$-module 
\[(M_1 \ox_{Z(A)} \io M_2)\ox_{A\ox_{Z(A)} \io A}A,\]
where the superscript $\iota$ on $A$ and $M_2$ indicates that the left
action of $Z(A)$ is twisted by $\iota$. The form $h_1 * h_2$ is then
obtained by computing $T_\s \bullet (h_1 \ox \io h_2)$, where $T_\s$ is
the form
$A\x A \to Z(A), (x,y)\mapsto \Trd_A(\s(x)y)$, and $\bullet$ denotes the
product of forms from \cite[I, (8.1), (8.2)]{knus91}. 

In concrete terms this
means that for all $m_1, m_1' \in M_1$, $m_2, m_2' \in M_2$ and $a,a'\in A$
we have
\[
  h_1 * h_2 (m_1\ox m_2\ox a, m_1' \ox m_2'\ox a')=
  \Trd_A\bigl(h_1(m_1a, m_1' a') \s(h_2(m_2,m_2'))\bigr),
\]
cf. \cite[Equation~(4.8)]{A-U-Az-PLG}. 

The pairing
$*$ 
satisfies the following  ``pivot property'' for all hermitian forms 
$h_1, h_2, h_3$  over $(A,\s)$:
\begin{equation}\label{pivot}
  (h_1 * h_2) \ox_{Z(A)} h_3 \simeq  (h_3 * h_2 ) \ox_{Z(A)} h_1,
\end{equation}
cf. \cite[Theorem~4.9]{A-U-Az-PLG}. We also note that if 
$a,b \in \Sym(A^\x,\s)$, then by \cite[Proposition~4.9]{garrel-2023} or 
\cite[Lemma~4.11]{A-U-Az-PLG} we have
\[
  \qf{a}_\s * \qf{b}_\s \simeq \vf_{a,b,\s},
\]
where $\vf_{a,b,\s}(x,y):=\Trd_A(\s(x)ayb)$.

\begin{defn}
  Let $\alpha\in \Sper R$. If $\alpha\in \Nil[A,\s]$, we define 
  $\lambda_\alpha:=1$, otherwise we define $\lambda_\alpha$ to be the 
  index 
  of the central simple algebra $A\ox_R k(\alpha)$ 
  (i.e., the degree of its division algebra part). 
\end{defn}

With notation as in Lemma~\ref{clopen-sets} we have for 
$\alpha \in \Sper R \sm  \Nil[A,\s]$ that $\lambda_\alpha=1$ on $C_4$ and $C_5$,
and $\lambda_\alpha=2$ on $C_6$.
In particular, the map $\lambda_\bullet: \Sper R \to \{1,2\}$ is continuous
by Lemma~\ref{clopen-sets}.

\begin{thm}\label{sig-mult}
  Let $(M_1, h_1)$ and $(M_2, h_2)$ be hermitian forms over $(A,\s)$, and let 
  $\alpha \in \Sper R$.
  Then  
  \[\sign_\alpha (h_1 * h_2) = \lambda_\alpha^2 \cdot\sign^\CM_\alpha h_1 \cdot 
  \sign^\CM_\alpha h_2,\]
  where $\sign_\alpha$ is defined in Section~\ref{quad-et} if $Z(A)\not=R$
  and is the usual
  signature of quadratic forms otherwise, and 
  the same Morita equivalence is used in the
  computation  of the
  signatures on the right hand side, cf. Remark~\ref{choice-Morita}.
\end{thm}

\begin{proof}
  Since the signatures on both sides of the equality are computed by extending
  sca\-lars to the real closed field $k(\alpha)$, we may assume that $(A,\s)$ is
  a central simple $F$-algebra with involution, where $F$ is a real closed
  field, and that $\alpha$ is the unique ordering on $F$.
  \smallskip
  
  \noindent\emph{Fact:} If one of $h_1$, $h_2$ is hyperbolic or the zero form,
  then $h_1 * h_2$ is hyperbolic or the zero form, respectively. 
  \smallskip

  \noindent\emph{Proof of the Fact:}   Assume that $h_1$ is hyperbolic.
  Let $U_1$ be a lagrangian of $h_1$, cf. \cite[p~19]{knus91}. 
  Then $U:=U_1\ox_S \io M_2$ is a 
  lagrangian of
  $h_1\ox \io h_2$, and $U \ox_{A\ox_S \io A} A$ is a lagrangian of
  $h_1 * h_2= T_\s \knp (h_1\ox \io h_2)$, which is then hyperbolic. 
  If $h_2$ is hyperbolic, the proof is similar. The statement about the
  zero form is clear from the definition of $h_1 * h_2$. This proves the fact.
  \smallskip
  
  By \cite[Proposition~A.3]{A-U-PS} we can write 
  $h_i\simeq h_i^\ns \perp\zeta_i$ ($i=1,2$),
  where $h_i^\ns$ is nonsingular and $\zeta_i$ is the zero form (defined on a
  module of appropriate
  rank), both unique up to isometry. 
  Recalling that $*$ respects isometries and orthogonal sums, cf. 
  \cite[Corollary~4.8]{A-U-Az-PLG}, and using the Fact, it follows
  that $h_1 * h_2 \simeq h_1^\ns * h_2^\ns \perp \zeta$, for some zero form
  $\zeta$. Therefore, 
  since any zero form has signature zero, we may
  assume for the purpose of computing 
  signatures,  that $h_1$ and $h_2$ are both nonsingular.

  If $\alpha\in\Nil[A,\s]$, then $h_1$ and $h_2$
  are weakly hyperbolic by Pfister's local-global
  principle for hermitian forms, cf. \cite[Theorem~4.1]{L-U} since
  $\alpha$ is the unique ordering on $F$. 
  By the Fact, $h_1 * h_2$ is weakly hyperbolic, and the claim then
  follows since the signatures of $h_1$, $h_2$ and
  $h_1 * h_2$ are all zero. 
  
  Hence we may assume that $\alpha\not\in\Nil[A,\s]$, and thus, by the
  Skolem-Noether theorem (cf. \cite[Propositions~2.7 and 2.18]{BOI}), 
  that
  $(A,\s)\cong   (M_n(D),\Int(\Psi) \circ \vt^t)$, where
  the notation is as follows: $(D,\vt)\in\{ (F, \id), (F(\sqrt{-1}),\iota), 
  ((-1,-1)_F,\gamma)\}$; $\iota$ and $\gamma$ denote conjugation;
  $\vt^t$ denotes conjugate transposition $(m_{ij})\mapsto (\vt(m_{ji}))$;  
  $\Psi\in M_n(D)$ is an
    invertible matrix that satisfies $\vt^t(\Psi)=\ve\Psi$ for some
    $\ve\in \{-1,1\}$. 
  By \cite[Second half of p. 499]{A-U-prime} 
  or \cite[Remark~6.2]{A-U-pos}, 
  and since $\alpha \not \in \Nil[A,\s]$, we have $\ve = 1$.   
  Without loss of generality we may assume that this isomorphism of algebras
  with involution is an equality, $(A,\s)= (M_n(D),
  \Int(\Psi) \circ \vt^t)$.
    
  Since $\alpha$ is the unique ordering on $F$,  we omit it as an index in the
  signature calculations that follow.

  By the ``weak diagonalization lemma'' \cite[Lemma~2.2]{A-U-pos}, 
  there are $N, r,
  s \in \N$ such that $N \x h_1 \simeq \qf{a_1, \ldots, a_r}_\s$ and $N \x h_2
  \simeq \qf{b_1, \ldots, b_s}_\s$, where $a_1,\ldots, a_r, b_1,\ldots, b_s \in
  \Sym(A^\x, \s)$. Recall from \cite[Proposition~3.6]{A-U-Kneb} that
  the $\CM$-signature is $\Z$-linear and from \cite[Corollary~4.8]{A-U-Az-PLG} that
  the pairing $*$ commutes with orthogonal sums in each component. It follows
  that
    \begin{align*}
      \sign (h_1 * h_2) &= \lambda_\alpha^2 \cdot \sign^\CM h_1 \cdot 
      \sign^\CM  h_2 \\
        &\Leftrightarrow 
      N^2 \cdot \sign (h_1 * h_2) =  \lambda_\alpha^2 \cdot (N\cdot \sign^\CM h_1) 
      (N\cdot \sign^\CM h_2) \\
        &\Leftrightarrow 
      \sign ((N \x h_1) * (N \x h_2)) = \lambda_\alpha^2 \cdot\sign^\CM (N \x h_1) 
      \cdot \sign^\CM
        (N \x h_2) \\
        &         \Leftrightarrow 
      \sign (\qf{a_1, \ldots, a_r}_\s * \qf{b_1, \ldots, b_s}_\s) =  \\
      & \hspace*{10em}
      \lambda_\alpha^2 \cdot
      \sign^\CM \qf{a_1, \ldots, a_r}_\s \cdot \sign^\CM \qf{b_1, \ldots, 
      b_s}_\s \\
        &\Leftrightarrow 
      \sign (\bigperp_{i,j} \qf{a_i}_\s * \qf{b_j}_\s) = \lambda_\alpha^2 
      \sum_{i,j}
      \sign^\CM \qf{a_i}_\s \cdot \sign^\CM \qf{b_j}_\s \\
        &\Leftrightarrow 
      \sum_{i,j} \sign (\qf{a_i}_\s * \qf{b_j}_\s) = \lambda_\alpha^2 \sum_{i,j}
      \sign^\CM \qf{a_i}_\s \cdot \sign^\CM \qf{b_j}_\s.
    \end{align*}
    Therefore, it suffices to show that
\begin{equation}\label{sign-star}
      \sign  \qf{a}_\s * \qf{b}_\s = \lambda_\alpha^2\cdot
          \sign^\CM \qf{a}_\s \cdot \sign^\CM \qf{b}_\s,
\end{equation}
    for all $a,b \in \Sym(A,\s)$ (in fact, it suffices to take 
    $a,b \in \Sym(A^\x,\s)$).

  Now recall  
  that $\qf{a}_\s * \qf{b}_\s
  \simeq\vf_{a,b,\s}$, where $\vf_{a,b,\s}(x,y) := \Trd_A(\s(x)ayb)$,
  cf. \cite[Lemma~4.11]{A-U-Az-PLG}. Since 
  $\s=\Int(\Psi)\circ \vt^t$, we have
    \begin{align*}
      \vf_{a,b,\s}(x,y) &= \Trd_A(\s(x)ayb)\\
                        &= \Trd_A(\Psi\vt^t(x) \Psi^{-1} ayb)\\
                        &= \Trd_A(\vt^t(x) (\Psi^{-1} a)y(b\Psi)),
    \end{align*}
  and thus $\qf{a}_\s * \qf{b}_\s \simeq \qf{\Psi^{-1} a}_{\vt^t} * 
  \qf{b\Psi}_{\vt^t} $.  Since $\qf{b\Psi}_{\vt^t} \simeq
  \qf{\vt^t(\Psi^{-1})  (b\Psi)\Psi^{-1}}_{\vt^t}=
  \qf{\Psi^{-1} b}_{\vt^t}$, 
  we conclude that  $\qf{a}_\s * \qf{b}_\s \simeq 
  \qf{\Psi^{-1} a}_{\vt^t} * \qf{\Psi^{-1} b}_{\vt^t} $.
  
  On the right-hand side of \eqref{sign-star} we can apply the 
  ``scaling-by-$\Psi^{-1}$'' Morita equivalence 
  and obtain
  $\sign^\CM\qf{a}_\s = \delta \sign^\CM\qf{\Psi^{-1} a}_{\vt^t}$ and
  $\sign^\CM\qf{b}_\s = \delta \sign^\CM\qf{\Psi^{-1} b}_{\vt^t}$, where 
  $\delta\in\{-1,1\}$, cf. \cite[Proposition~3.4]{A-U-Kneb}. 

  Thus we may assume that
  $(A,\s) = (M_n(D), \vt^t)$, and that
  $a,b\in \Sym(M_n(D), \vt^t)$ and $\s=\vt^t$ in \eqref{sign-star}.
    By the Principal Axis Theorem (cf. \cite[Appendix A]{A-U-pc-gauge} and 
  \cite[Corollary~6.2]{Zhang}), we may further assume that $a$ and $b$ are
  diagonal matrices in $M_n(F)$, say $a=\diag(a_1,\ldots,a_n)$ and
  $b=\diag(b_1,\ldots,b_n)$. From the computation of $\CM$-signatures it then
  follows that
  \[
    \sign^\CM \qf{a}_{\vt^t} \cdot \sign^\CM \qf{b}_{\vt^t}
    =\sign \qf{a_1,\ldots,a_n}\cdot \sign \qf{b_1,\ldots, b_n}.
  \]
  Let $\{e_{pq}\}$ denote the canonical $D$-basis of $M_n(D)$.
  We now consider three cases, that correspond to $\alpha \in C_4$, $C_5$
  and $C_6$, respectively:
  
  \begin{enumerate}[a)] 
    \item $(D,\vt)=(F,\id)$. A direct computation shows that
    $\{e_{pq}\}$  is an orthogonal basis for the symmetric bilinear form
    $\vf_{a,b, \vt^t}$ over $F$, and that the matrix of $\vf_{a,b,\vt^t}$ 
    relative to  
    this basis is
      \[\diag(a_1b_1, \ldots, a_1b_n, \ldots, a_nb_1, \ldots, a_nb_n).\]
    Therefore,
      \[
        \qf{a}_{\vt^t} *\qf{b}_{\vt^t}   
        \simeq \qf{a_1b_1, \ldots, a_1b_n, \ldots, a_nb_1, \ldots, a_nb_n}.
      \]
    
    \item $(D,\vt)=(F(\sqrt{-1}),\iota)$. A direct computation shows that
    $\{e_{pq}\}$  is an orthogonal basis for the hermitian form
    $\vf_{a,b, \vt^t}$ over $(F(\sqrt{-1}),\iota)$, 
    and that the matrix of $\vf_{a,b,\vt^t}$ relative to  
    this basis is
      \[\diag(a_1b_1, \ldots, a_1b_n, \ldots, a_nb_1, \ldots, a_nb_n).\]
    Therefore,
      \[
        \qf{a}_{\vt^t} *\qf{b}_{\vt^t}   
        \simeq \qf{a_1b_1, \ldots, a_1b_n, \ldots, a_nb_1, \ldots, 
        a_nb_n}_{\iota}.
      \]
    
    \item $(D,\vt)=((-1,-1)_F,\gamma)$. Let $\{1,i,j,k\}$ be the standard $F$-basis
    of $D$. 
    A direct computation shows that
    $\{ e_{pq}, i\cdot e_{pq},j\cdot e_{pq},k\cdot e_{pq} \}$  
    is an orthogonal basis for the symmetric bilinear form
    $\vf_{a,b, \vt^t}$ over $F$, and that the matrix of $\vf_{a,b,\vt^t}$ 
    relative to  
    this basis is the block diagonal matrix
      \[\diag( a_1b_1\Delta, \ldots,  a_1b_n \Delta, \ldots,  a_nb_1 \Delta, 
      \ldots,  a_nb_n \Delta),\]
      where $\Delta=\diag(2,2,2,2)$.
    Therefore,
      \[
        \qf{a}_{\vt^t} *\qf{b}_{\vt^t}   
        \simeq \qf{2,2,2,2}\ox
        \qf{a_1b_1, \ldots, a_1b_n, \ldots, a_nb_1, \ldots, a_nb_n}.
      \]
  \end{enumerate}
    
    Since
    \begin{align*}
      \sign \qf{a_1b_1, \ldots, a_1b_n, \ldots, a_nb_1, \ldots, a_nb_n} &=
      \sign (\qf{a_1,\ldots, a_n}\ox \qf{b_1,\ldots, b_n})\\
      &=\sign \qf{a_1,\ldots, a_n} \cdot \sign \qf{b_1,\ldots, b_n},
    \end{align*}
    we obtain 
    \begin{align*}
      \sign \qf{a}_{\vt^t} *\qf{b}_{\vt^t} &=\lambda_\alpha^2 \cdot \sign 
        \qf{a_1,\ldots, a_n} \cdot \sign \qf{b_1,\ldots, b_n}\\
      &=\lambda_\alpha^2 \cdot \sign^\CM \qf{a}_{\vt^t} \cdot \sign^\CM 
      \qf{b}_{\vt^t}.\qedhere   
    \end{align*}
\end{proof}

\section{Continuity of some total signature maps}

We recall the case of quadratic forms over a commutative ring:

\begin{prop}\label{qf-cont}
  Let $q$ be a quadratic form over $R$, defined on a finitely
  generated projective $R$-module $M$. Then the total signature map
  $\sign_\bullet q : \Sper R
  \rightarrow \Z$ is continuous, where $\Sper R$ is equipped with the
  constructible topology. If $q$ is nonsingular, it is also continuous if $\Sper
  R$ is equipped with the Harrison topology.
\end{prop}
\begin{proof}
  The second statement is in \cite[Section~1.3]{mahe82}. 
  We could not find a reference for the first statement,
  so we provide a proof. Since $M$ is finitely generated projective, there
  are elements $f_1, \ldots, f_k \in R$ such that $\Spec R = D(f_1) \cup \cdots
  \cup D(f_k)$ (using the notation $D(f_i)=\{\fp \in \Spec R \mid f_i\not\in 
  \fp\}$)  and the modules $M_{f_i}$ are free and finitely generated (see
  for instance \cite[Chapitre II, \S5.2, Th\'eor\`eme~1]{bourbaki-AC-1-4}).

  Since the canonical map $\Supp:\Sper R \rightarrow \Spec R$ is continuous (for the
  Harrison topology on $\Sper R$, cf. \cite[Proposition~3.3.3]{KSU}), 
  the sets $\Supp^{-1}(D(f_i))$ are open and so  
  we can assume that $M$ is free (by replacing $M$ by $M_{f_i}$). Let
  $P_q$ be the matrix of $q$ in a fixed basis of $M$ over $R$, and let
  $\chi_q(X) = X^n + u_1X^{n-1} + \cdots + u_n$ be the characteristic polynomial
  of $P_q$. By \cite[Remark~1.3.25(2)]{S24} the signature of $q$ at $\alpha
  \in \Sper R$ is given by $\Var_\alpha(1,u_1,u_2 \ldots, u_n) - \Var_\alpha(1,
  -u_1, u_2, \ldots, (-1)^nu_n)$, where $\Var_\alpha(u_1, \ldots, u_n)$
  represents the number of sign changes with respect to $\alpha$ in the sequence
  $u_1, \ldots, u_n$ (after removing the elements of sign zero at $\alpha$).

  Therefore, for $t \in \N\cup\{0\}$, the set $\{\alpha \in \Sper R \mid 
  \sign_\alpha q =
  t\}$ can be expressed as a boolean combination of sets of the form $\{\alpha
  \in \Sper R \mid \sgn_\alpha u_i = \ve_i\}$ (where $\sgn_\alpha$ denotes
  the sign at $\alpha$), for some $\ve_i \in
  \{-1,0,1\}$, and so is constructible.
\end{proof}

Next, we consider the case of signatures of hermitian forms over a
quadratic \'etale $R$-algebra $S$ with standard involution $\vt$, cf.
Section~\ref{quad-et}.

\begin{prop}\label{pizza}
  Let $h$ be a hermitian form over $(S,\vt)$. The total signature map 
  $\sign_\bullet h : \Sper
  R \rightarrow \Z$ is continuous for the constructible topology on $\Sper R$.
  Furthermore, if $h$ is nonsingular, then  $\sign_\bullet h $ is continuous
  for the Harrison topology on $\Sper R$.
\end{prop}

\begin{proof}
    According to equation~\eqref{signcentre1}, we have $\sign_\alpha h = 0$ if
    $S \ox_R \kappa(\alpha) \cong \kappa(\alpha) \x \kappa(\alpha)$, i.e., if
    $\alpha \in \Nil[S,\vt]$ (since $\vt$ is of the second kind by
    definition). Since $\Nil[S,\vt]$ is clopen (for the Harrison topology) by
    Theorem~\ref{nil-clopen}, it suffices to consider $\sign_\bullet h$ on the
    set $\Sper R \setminus \Nil[S,\vt]$. On this set $\sign_\bullet h$ is
    defined by equation~\eqref{signcentre2} and the result follows by
    Proposition~\ref{qf-cont}.
\end{proof}

As a consequence we obtain the following result for hermitian forms over 
$(A,\s)$:

\begin{cor}\label{abs-val-cont} 
  Let $h$ be a hermitian form over $(A,\s)$. Then 
  \[|\sign^\CM_\bullet h| = \frac{1}{\lambda_\bullet} 
  \sqrt{\sign_\bullet (h *h)}\]
  on $\Sper R$. In particular, $|\sign^\CM_\bullet h|$ is
  continuous for the constructible topology on $\Sper R$,
  and for the Harrison topology if $h$ is nonsingular.
\end{cor}

\begin{proof}
  The first statement follows directly from Theorem~\ref{sig-mult}. 
  The second statement from the fact that $\lambda_\bullet$ is continuous
  and Proposition~\ref{pizza}, since  $h * h$ is 
  a hermitian form over $(Z(A),\iota)$, which is nonsingular if $h$ is
  (as recalled in Section~\ref{secstarmult}). 
\end{proof}

\section{Reference forms and $\eta$-signatures}
\label{sec-eta-sign}

As observed after Definition~\ref{def:sig_h}, the sign of the $\CM$-signature 
at $\alpha \in \Sper R\sm \Nil[A,\s]$ can be changed by taking a
different Morita equivalence. One way to make a choice of sign is to fix a
hermitian form $\eta$ whose $\CM$-signature at $\alpha$
is known to be nonzero, and to choose a Morita equivalence that makes the
$\CM$-signature at $\alpha$ of  $\eta$ positive. 
Therefore: 

\begin{defn}
  A hermitian form $\eta$ over $(A,\s)$  is called a 
  \emph{reference form  for $(A,\s)$} if 
  $\sign^\CM_\alpha \eta\not=0$ for all $\alpha \in \Sper R\sm
      \Nil[A,\s]$.
\end{defn}

In this section we will show that reference forms always exist for
Azumaya algebras with involution over  commutative rings.

\begin{lemma}\label{lem:diag}
  Let $\fp \in \Spec R$ and let $\kappa(\fp)$ be the field of fractions of
  $R/\fp$. 
  Any diagonal hermitian form over $(A\ox_R  \kappa(\fp), \s\ox 
  \id_{\kappa(\fp)})$ is of
  the form $h\ox_R \kappa(\fp)$, for some 
  diagonal hermitian form $h$ over $(A,\s)$.
\end{lemma}

\begin{proof}
  We write $\s(\fp):= \s\ox 
  \id_{\kappa(\fp)}$.
  It suffices to prove the statement for the form $\vf:=\qf{a}_{\s(\fp)}$,
  where $a=\sum_i a_i\ox ({\ovl{r_i}}/{\ovl{s_i}})$ and $\ovl{r_i},
  \ovl{s_i}$ are in $(R/\fp)\setminus \{0\}$.  
  After multiplying $a$ by the (hermitian) square $1\ox (\prod_i \ovl{s_i})^2$,
  we obtain $\vf \simeq \qf{\sum_i
  a_i \ox \ovl{r'_i} }_{\s(\fp)}$ for some $\ovl{r'_i} \in 
  R/\fp$. Now
  using that $1\ox \ovl{r'_i}=r'_i \ox 1$, we obtain $\vf \simeq
  \qf{(\sum_i a_i r'_i) \ox 1}_{\s(\fp)} = \qf{\sum_i a_i r'_i }_{\s}\ox
  \kappa(\fp)$.	
\end{proof}

\begin{lemma}\label{diag-ref}
  For every $\alpha \in \Sper R \setminus \Nil[A,\s]$, there is
  a (diagonal) hermitian form $\eta_\alpha$ over $(A,\s)$ such that 
  $\sign^\CM_\alpha \eta_\alpha  \not= 0$.
\end{lemma}

\begin{proof}
  Let $\alpha \in \Sper R \setminus \Nil[A,\s]$. 
  Then $\bar\alpha \not \in \Nil[A(\alpha), \s(\alpha)]$ by 
  \cite[Lemma~3.7]{A-U-Az-PLG}.
  By 
  \cite[Theorem~6.4]{A-U-Kneb} 
  there is a diagonal
  hermitian form $\vf_\alpha$ over $(A(\alpha), \s(\alpha))$ such that
  $\sign^{\CM_{\bar \alpha}}_{\bar \alpha} \vf_\alpha \not = 0$. Thus, by 
  Lemma~\ref{lem:diag} with $\fp=\Supp(\alpha)$, 
  there is a (diagonal) form $\eta_\alpha$ over $(A,\s)$ such that $\eta_\alpha
  \ox_R \kappa(\alpha) = \vf_\alpha$. Therefore, $\sign^\CM_\alpha \eta_\alpha =
  \sign^{\CM_{\bar \alpha}}_{\bar \alpha} \vf_\alpha \not = 0$.
\end{proof}

\begin{lemma}\label{like-Mahe}
  Let $U \subseteq \Sper R$ be clopen for the constructible topology.  Then
  there is a quadratic form $q$ over $R$ and $k \in \N$ such that 
  $\sign_\bullet q = 0$
  on $U$ and $\sign_\bullet q = 2^k$ on $(\Sper R) \setminus U$.
\end{lemma}
\begin{proof}
  By \cite[Theorem~3.4.5]{KSU}, 
  $U$ is a constructible subset of $\Sper R$, i.e., a
  finite boolean combination of sets of the form $\mathring H(a)$ for $a \in
  R$. We proceed by induction on this boolean combination:
  \begin{enumerate}[a)]   
  \item Let  $U = \mathring H(a)$. Consider $q := \qf{a}$. Then 
    $\sign_\bullet q = 1$ on
      $\mathring H(a)$, $0$ on $\{\alpha \in \Sper R \mid a \in \alpha \cap 
      -\alpha\}$ and $-1$ on $\{\alpha \in \Sper R \mid a \in -\alpha \setminus 
      \alpha\}$. Inspired by the observation that the polynomial $-x^2+x+2$ 
      takes value $0$ at $x=-1$, and value $2$ at $x=0$ and $x=1$, we let
      $q' := -q \ox q \perp q \perp \qf{1,1}$. Then
      $\sign_\bullet q' = 0$ on $\mathring H(a)$ and $\sign_\bullet q'=2$ on
      $(\Sper R) \setminus \mathring H(a)$.
      \medskip

 \item Let $U = V_1 \cup V_2$. By induction there are quadratic forms $q_1, q_2$
      such that $\sign_\bullet q_i$ equals $0$ on $V_i$ and $2^{k_i}$ on 
      $(\Sper R) \sm V_i$ for 
      $i=1,2$. Let
      $q' := q_1 \ox q_2$. Then $\sign_\bullet q'$ equals $0$ on $V_1 \cup V_2$ and $2^{k_1 +
      k_2}$ on $(\Sper R) \setminus (V_1 \cup V_2)$.
      \medskip

\item Let $U = (\Sper R) \setminus V$. By induction there is a quadratic form
      $q$ such that $\sign_\bullet q$ equals $0$ on $V$ and $2^k$ on $(\Sper R) 
      \setminus V$.
      Let $q' := 2^k \x \qf{1} \perp -q$. Then $\sign_\bullet q'$ equals 
      $0$ on $U$ and $2^k$
      on $(\Sper R) \setminus U$.\qedhere
 \end{enumerate}
\end{proof}

\begin{thm}\label{ref-form}
  There exists a reference form $\eta$ for $(A,\s)$ such that
  $|\sign^\CM_\bullet \eta|$ is a constant nonzero function on $\Sper R$.
\end{thm}

\begin{proof}
  Let $\alpha \in \Sper R \setminus \Nil[A,\s]$. By Lemma~\ref{diag-ref}, there
  exists a hermitian form $\eta_\alpha$ over $(A,\s)$ such that
  $s_\alpha:=|\sign^\CM_\alpha \eta_\alpha| > 0$. By
  Corollary~\ref{abs-val-cont}, the set
  \[
    U_\alpha := |\sign^\CM_\bullet \eta_\alpha|^{-1}(s_\alpha) 
  \]
  is clopen in $\Sper R$ for the constructible topology.
  Therefore, by compactness of $\Sper R \setminus \Nil[A,\s]$ (cf.
  Theorem~\ref{nil-clopen}), there are clopen sets $U_i$ and hermitian forms
  $\eta_i$ over $(A,\s)$, for $i=1, \ldots, k$ (for some $k \in \N$) such that
  $\Sper R \setminus \Nil[A,\s] = U_1 \cup \cdots \cup U_k$ and
  $|\sign^\CM_\bullet \eta_i| =: s_i > 0$ on $U_i$. Up to removing their
  intersections, we can assume that $U_1, \ldots, U_k$ are disjoint. Moreover,
  up to replacing each hermitian form $\eta_i$ by sufficiently many copies of
  itself, we may assume that $s_1 = \cdots = s_k =: s$.

  By Lemma~\ref{like-Mahe}, there are quadratic forms $q_i$ over $R$ such that
  $\sign_\bullet q_i$ equals $2^{k_i}$ on $U_i$ and $0$ on $(\Sper R) \setminus
  U_i$. Again, up to replacing each quadratic form $q_i$ by sufficiently many
  copies of itself, we can assume that $\sign_\bullet q_i$ equals $2^{k}$ on
  $U_i$ and $0$ on $(\Sper R) \setminus U_i$.

  Therefore the hermitian form $\displaystyle\midperp_{i=1}^k q_i \ox \eta_i$
  has the required property, since its signature at any $\alpha \in \Sper R$ is
  $2^k s$.
\end{proof}

\begin{defn}\label{eta-sign}
  Let $\eta$ be a reference form for $(A,\s)$ and let $h$ be a hermitian form
  over $(A,\s)$. For $\alpha \in \Sper R$ we define the \emph{$\eta$-signature
  of $h$ at $\alpha$} by 
  \[\sign^\eta_\alpha h := \sgn (\sign^{\CM}_\alpha \eta) \cdot \sign^{\CM}_\alpha h,\]
  where $\sgn: \Z\to \{-1,1\}$ denotes the sign function.
\end{defn}

\begin{rem}\label{obvob}
  Observe that:
  \begin{enumerate}[{\quad\rm (1)}]
    \item We have $\sign^\eta_\alpha \eta >0$ whenever $\alpha\in \Sper R\sm
    \Nil[A,\s]$, and that $\sign^\eta_\alpha$ is a special case of
    $\sign^{\CM}_\alpha$ since it can be obtained from it for a well-chosen
    Morita equivalence.
  
    \item If $\sign^\CM_\alpha h \not =
    0$, then
    \[\sign^\eta_\alpha h > 0 \Leftrightarrow (\sign^\CM_\alpha \eta 
    \text{ and }
    \sign^\CM_\alpha h \text{ have the same sign}).\]
  \end{enumerate}
\end{rem}

\begin{prop}
  With notation as in Definition~\ref{eta-sign}, the value of 
  $\sign^\eta_\alpha h$
  is independent of the choice of the Morita equivalence used in the computation
  of $\sign^\CM_\alpha h$.
\end{prop}

\begin{proof}
  The statement is trivially true if $\alpha\in \Nil[A,\s]$, so we may assume
  that
  $\alpha \in \Sper R \sm \Nil[A,\s]$.
  Let $\CN$ be another Morita equivalence in the definition of 
  $\sign^\eta_\alpha h$. Then, using that the signature is computed by
  first extending scalars to $\kappa(\alpha)$, and then that 
  $(A(\alpha), \s(\alpha))$ is a central simple algebra with involution, 
  we
  obtain by  \cite[Lemma~3.8]{A-U-Kneb} that
  \begin{align*}
    \sgn (\sign^{\CN}_\alpha \eta) \cdot \sign^{\CN}_\alpha h 
        &= \sgn (\sign^{\CN_{\bar\alpha}}_{\bar\alpha} (\eta\ox \kappa(\alpha)))
         \cdot 
          \sign^{\CN_{\bar\alpha}}_{\bar\alpha} (h \ox \kappa(\alpha)) \\
        &= \sgn (\sign^{\CM_{\bar\alpha}}_{\bar\alpha} (\eta\ox \kappa(\alpha))) \cdot 
        \sign^{\CM_{\bar\alpha}}_{\bar\alpha} (h \ox \kappa(\alpha))\\
        &= \sgn (\sign^{\CM}_\alpha \eta) \cdot \sign^{\CM}_\alpha h. \qedhere
  \end{align*}
\end{proof}

\begin{prop}\label{change-rcf}
  Assume we have ring morphisms $f:R\to L$ and $g:R\to N$
  where $L$ and $N$ are real closed fields such that
  $f^{-1}(L_{\geq 0}) = g^{-1}(N_{\geq 0})=:\alpha  \in \Sper R$.
  Then, for any hermitian form $h$
  over $(A,\s)$, we have
  \[\sign^{\eta \ox_f L} (h \ox_f L) = \sign^{\eta \ox_g N} (h \ox_g N)\]
  with respect to the unique orderings on $L$ and $N$.
\end{prop}

\begin{proof}
  Since
  $\Supp (\alpha) \subseteq \Ker f$, $f$ induces a map $\bar f : \kappa(\alpha)
  \rightarrow L$ and, similarly, $g$ induces a map $\bar g : \kappa(\alpha)
  \rightarrow N$, so that  both triangles in the following diagram commute:
  \[\xymatrix{
    R \ar[dr]^{\rho_\alpha} \ar@/^1pc/[drr]^f \ar@/_1pc/[ddr]_g & & \\
      & \kappa(\alpha) \ar[r]^{\bar f} \ar[d]_{\bar g} & L \\
      & N & 
    }\]  
  By definition of signature, we have $\sign^\eta_\alpha h = \sign^{\eta \ox
  \kappa(\alpha)}_{\bar \alpha} (h \ox_R \kappa(\alpha))$. Applying 
  \cite[Theorem~2.9]{A-U-sign-pc}, we obtain
  \[\sign^{\eta \ox \kappa(\alpha)}_{\bar \alpha} (h \ox_R \kappa(\alpha)) =
  \sign^{\eta \ox \kappa(\alpha) \ox_{\bar f} L} \bigl((h \ox_R \kappa(\alpha))
  \ox_{\bar f} L\bigr) = \sign^{\eta \ox_f L} (h \ox_f L),\]
  and
  \[\sign^{\eta \ox \kappa(\alpha)}_{\bar \alpha} (h \ox_R \kappa(\alpha)) =
  \sign^{\eta \ox \kappa(\alpha) \ox_{\bar g} N} \bigl((h \ox_R \kappa(\alpha))
  \ox_{\bar g} N\bigr) = \sign^{\eta \ox_g N} (h \ox_g N),\]
  proving the result.
\end{proof}

\section{Continuity of the total signature map}
\label{sec-cont}

\begin{thm}\label{sign-cont}
  Let $\eta$ be a reference form for $(A,\s)$. Then for any hermitian form $h$
  over $(A,\s)$ the total signature map $\sign^\eta_\bullet h : \Sper R
  \rightarrow \Z$ is continuous with respect to the constructible topology
  on $\Sper R$ and the discrete topology on $\Z$. 

  If $\eta$ and $h$ are nonsingular, then the total signature map
  $\sign^\eta_\bullet h : \Sper R \rightarrow \Z$ is continuous 
  with respect to the Harrison topology on $\Sper R$ and the discrete topology 
  on $\Z$.
\end{thm}

\begin{proof}
  We prove the first statement (for the constructible topology), the same proof
  \emph{mutatis mutandis}  gives the second one.

  By definition of $\Nil[A,\s]$ and since $\Nil[A,\s]$ is clopen, cf.
  Theorem~\ref{nil-clopen}, it suffices to show that the map
  $\sign^\eta_\bullet h$ is continuous on $\Sper R \setminus \Nil[A,\s]$. The
  map $\sign^\eta_\bullet \eta$ only takes a finite number of values on $\Sper
  R \setminus \Nil[A,\s]$ (see Lemma~\ref{sign-bounded}), say $k_1, \ldots,
  k_\ell$, which are all positive by Remark~\ref{obvob}(1).
  Let 
  \begin{align*}
    O_i &:= \{\alpha \in (\Sper R)
      \setminus \Nil[A,\s] \mid \sign^\eta_\alpha \eta = k_i\}\\
      &\phantom{:}= \{\alpha \in
        (\Sper R) \setminus \Nil[A,\s] \mid |\sign^{\CM}_\alpha \eta| 
        = k_i\}.
  \end{align*}
  Then 
  $O_i$ is open
  by Corollary~\ref{abs-val-cont} and $(\Sper R) \setminus \Nil[A,\s]$ 
  is
  the disjoint union of the sets $O_1, \ldots, O_\ell$. Therefore it suffices to
  show that the map $\sign^\eta_\bullet h$ is continuous from $O_i$ to $\Z$. 
  But, for $\alpha \in O_i$ and $r\in\Z$, we have
  \begin{align*}
    \sign^\eta_\alpha h = r &\Leftrightarrow k_i \sign^\eta_\alpha h = k_i r \\
      &\Leftrightarrow \sign^\eta_\alpha (k_i \times h) = \sign^\eta_\alpha(r \times 
      \eta) \\
      &\Leftrightarrow \sign^\eta_\alpha (k_i \times h - r \times \eta) = 0 \\
      &\Leftrightarrow |\sign^{\CM}_\alpha (k_i \times h - r \times \eta)| = 0,
  \end{align*}
  so that
  \[(\sign^\eta_\bullet h)^{-1}(r) \cap O_i = \{\alpha \in O_i  \mid
  |\sign^{\CM}_\alpha (k_i \times h - r \times \eta)| = 0\},\]
  which is open by Corollary~\ref{abs-val-cont}.
\end{proof}

\begin{rem}
  At the moment we do not know if it is always possible to find a reference form
  $\eta$ such that $\sign_\bullet^\eta h$ is continuous for the Harrison
  topology, for all nonsingular hermitian forms $h$ over $(A,\s)$ (it is
  certainly possible to do this in some special cases).
  The following example shows that it is possible to have a reference form 
  $\eta$
  such that $|\sign_\bullet^\CM \eta| = k > 0$ on $\Sper R$, and a nonsingular
  hermitian form $h$ on $(A,\s)$ such that $\sign_\bullet^\eta h$ is not
  continuous for the Harrison topology.
\end{rem}

\begin{exa}
  Let $h$ be a nonsingular hermitian form over $(A,\s)$. We define
  \begin{align*}
    \{h = 0\} &:=\{\alpha \in\Sper R \mid \sign^\CM_\alpha h = 0\}, 
    \text{ and}\\
    \{h \not = 0\} &:=\{\alpha \in\Sper R \mid \sign^\CM_\alpha h 
    \not =0\}.
  \end{align*}
  Since $h$ is nonsingular, both $\{h = 0\}$ and $\{h \not = 0\}$ are clopen for
  the Harrison topology on $\Sper R$, cf. Corollary~\ref{abs-val-cont}.
  
  Let $U$ be clopen for the constructible topology such that $U \cap \{h \not =
  0\}$ is not clopen for the Harrison topology. (For instance, take $(A,\s) =
  (M_n(R),t)$ and $h = \qf{1}_\s$. Then $|\sign^\CM_\alpha h|=n$ for every
  $\alpha\in \Sper R$, and we can take for $U$ any subset of $\Sper R$ that is
  clopen for the constructible topology, but not for the Harrison topology.)
  
  We construct a reference form $\eta$ such that $|\sign^\CM_\bullet \eta| = k
  > 0$ on $\Sper R$, but $\sign^\eta_\bullet h$ is not continuous for the
  Harrison topology.

  Let $k_1, \ldots, k_r$ be the non-zero values taken by $|\sign^\CM_\bullet
  h|$ on $\Sper R$, and $O_i := |\sign^\CM_\bullet h|^{-1}(k_i)$. Let $k := k_1
  \cdots k_r$. By \cite[Theorem~3.2]{mahe82} there are $s\in \N$ and
  quadratic forms $q_1,\ldots, q_r$ over $R$ such that
  $\sign_\bullet q_i = 2^s k/k_i$ on $O_i$ and $0$ elsewhere. 
  If we define $q_0 := q_1 \perp \ldots \perp q_r$ we have
  \[\sign_\bullet q_0 > 0 \text{ and }
  |\sign^\CM_\bullet (q_0 \ox h)| =2^s k \text{ on } \{h \not = 0\}.\]

  By Lemma~\ref{like-Mahe}, there is a quadratic form $q$ over $R$ such that
  $\sign_\bullet q = 2^\ell$ on $U$ and $-2^\ell$ on $(\Sper R) \setminus U$.
  Then $|\sign^\CM_\bullet (q \ox q_0 \ox h)|=k2^{s+\ell}$ on $\{h \not = 0\}$. 
  Let
  $\eta_0$ be a hermitian form over $(A,\s)$ such that $|\sign^\CM_\bullet
  \eta_0| = c>0$ on $\{h=0\}$ and $0$ on $\{h \not = 0\}$ (it can be obtained
  by multiplying the form obtained in Theorem~\ref{ref-form} by a suitable
  quadratic form as in Lemma~\ref{like-Mahe}). Defining
  \[\eta := c\x q \ox q_0 \ox h \perp k2^{s+\ell}\x \eta_0,\]
  we have $|\sign^\CM_\bullet \eta| = ck2^{s+\ell}$ on $\Sper R$. In particular
  $\eta$ is a reference form for $(A,\s)$.
  \smallskip
  
  \noindent\emph{Claim:} $\sgn(\sign^\eta_\bullet h)$ is not continuous for the Harrison
  topology. In particular $\sign^\eta_\bullet h$ is not continuous for the
  Harrison topology.
  \smallskip
   
  \noindent\emph{Proof of the claim:}
  If $\alpha \in \{h \not = 0\} \cap U$, we have
      \[\sign^\CM_\alpha \eta = c(\sign_\alpha q) (\sign_\alpha q_0) 
      (\sign^\CM_\alpha h) =
      c'\sign^\CM_\alpha h,\]
      for some $c'>0$, so that $\sgn(\sign^\eta_\alpha h) = 1$ by
      Remark~\ref{obvob}(2).
  
 If $\alpha \in \{h \not = 0\} \setminus U$, we have
      \[\sign^\CM_\alpha \eta = c(\sign_\alpha q) (\sign_\alpha q_0) 
      (\sign^\CM_\alpha h) =
      - c'\sign^\CM_\alpha h,\]
      for some $c'>0$,
      so that $\sgn(\sign^\eta_\alpha h) = -1$  by Remark~\ref{obvob}(2).

  Therefore $(\sgn(\sign^\eta_\bullet h))^{-1}(1) \cap \{h \not = 0\} = U
  \cap \{h \not = 0\}$, which is not clopen in the Harrison topology. This
  proves the claim and completes the example.
\end{exa}

We clarify some properties of the $\eta$-signature. For central simple
algebras with involution over a field these already appeared as 
\cite[Proposition~3.3]{A-U-prime} in a slightly different context.
Proposition~\ref{h_zero} is stated for the constructible topology on $\Sper R$, but the statements can easily be formulated
for the Harrison topology in cases where the total signature maps are
continuous for this topology.

\begin{prop}\label{h_zero}
  Let $\eta$ be a reference form for $(A,\s)$. The topology on $\Sper R$ is
  always the constructible topology in this proposition.
    \begin{enumerate}[{\quad \rm(1)}]
    \item Let $h_0$ be such that $\sign^\eta_\bullet h_0>0$ on $\Sper R \sm 
    \Nil[A,\s]$. Then  $\sign^\eta = \sign^{h_0}$.

    \item Let $f$ be a continuous map from $\Sper R$ to $\{-1,1\}$ (with the 
    discrete topology).
     There exists a reference form $h_f$ over $(A,\s)$ such that
    $\sign^{h_f} = f \cdot \sign^\eta$.

    \item Let $\eta'$ be another  reference form for $(A,\s)$. There exists a 
    continuous map $f:\Sper R \to \{-1,1\}$ such that
    $\sign^{\eta'} = f \cdot \sign^\eta$.
    \end{enumerate}
\end{prop}

\begin{proof} 
(1) Let $h$ be a hermitian form over $(A,\s)$.
  Considering that for a given $\alpha \in \Sper R$, $\sign^\eta_\alpha$ is a 
  special case of
  $\sign^{\CM}_\alpha$, it follows from Definition~\ref{eta-sign} that
  \[\sign^{h_0}_\alpha h = \sgn(\sign^\eta_\alpha h_0)\cdot \sign^\eta_\alpha h=
      \begin{cases}
        \sign^\eta_\alpha h & \text{if } \alpha\in \Sper R \sm \Nil[A,\s],\\
        0 & \text{if } \alpha\in \Nil[A,\s].
      \end{cases}\]

(2) 
By Lemma~\ref{like-Mahe} there exists a quadratic form $q$ over $R$
such that $\sign q > 0$ on $f^{-1}(1)$ and
$\sign q < 0$ on $f^{-1}(-1)$. 
Let $h_f = q\ox \eta$.   Then for any hermitian form $h$ over $(A,\s)$ we have
(since $\sign^\eta_\alpha h_f = \sign_\alpha q \cdot \sign^\eta_\alpha \eta$,
cf. Remark~\ref{signprop}),
\begin{align*}
    \sign^{h_f}_\alpha h & = \sgn(\sign^\eta_\alpha h_f)\cdot 
    \sign^\eta_\alpha h\\ &=
      \begin{cases}
        \sign^\eta_\alpha h & \text{if } \sign^\eta_\alpha h_f > 0, 
        \text{ i.e., if } f(\alpha) = 1\\
        -\sign^\eta_\alpha h & \text{if } \sign^\eta_\alpha h_f < 0, 
        \text{ i.e., if } f(\alpha) = -1\\
      \end{cases}\\
    & = f(\alpha) \sign^\eta_\alpha h.
\end{align*}

(3) Let $\alpha\in \Sper R \sm \Nil[A,\s]$. By \cite[Proposition~3.4]{A-U-Kneb} 
there exists $\delta \in \{-1,1\}$ such that
$\sign^\eta_\alpha \eta= \delta \sign^{\eta'}_\alpha \eta$. Thus, 
$\delta = \sgn\bigl(\sign^{\eta'}_\alpha \eta\bigr)$, 
since $\sign^\eta_\alpha \eta >0$.
Define $f: \Sper R \to \{-1,1\}$ by
 \[f(\alpha) =
      \begin{cases}
        \sgn\bigl(\sign^{\eta'}_\alpha \eta\bigr) & \text{if } \alpha\in 
            \Sper R \sm \Nil[A,\s],\\
        1 & \text{if } \alpha\in \Nil[A,\s].
      \end{cases}\]
Then $f$ is continuous and $\sign^\eta = f\cdot \sign^{\eta'}$.      
\end{proof}

For $\alpha \in \Sper R$ we consider
\begin{align*}
  m_\alpha(A,\s) &:= \max \{\sign^\CM_\alpha \qf{a}_\s \mid a \in \Sym(A^\x,\s) 
  \}\\
  &\phantom{:}= \max \{\sign^\eta_\alpha \qf{a}_\s \mid a \in \Sym(A^\x,\s) 
  \}.
\end{align*}
This function was of crucial importance in \cite{A-U-Az-PLG} to establish the
links between elements of maximal signature and positive semidefinite forms,
cf. \cite[Proposition~5.4, Theorem~6.1]{A-U-Az-PLG}. However, we could only
determine its value for maximal elements of $\Sper R$, cf.
\cite[Corollary~3.19]{A-U-Az-PLG}. The following results clarify this.

\begin{lemma}\label{cont-spear}
  Let $f : \Sper R \rightarrow \Z$ be continuous, 
  where $\Sper R$ is equipped with the Harrison topology.   
  Let $\alpha \subseteq
  \beta \in \Sper R$. Then $f(\alpha) = f(\beta)$.
\end{lemma}
\begin{proof}
  Recall that $\alpha\subseteq \beta$ is equivalent to $\beta \in 
  \overline{\{\alpha\}}$, cf. 
  \cite[Definition~3.3.6 and Proposition~3.3.7]{KSU}.
  Let $k:=f(\beta)$ and consider the open set  
  $U=f^{-1}(k)$. Since $\beta\in U$, we have $\alpha\in U$. Hence,
  $f(\alpha) =k$.
\end{proof}

\begin{prop}\label{cst-spear}
  Let $\eta$ be a reference form for $(A,\s)$ and assume that $\eta$ is
  nonsingular. Let $\alpha\subseteq \beta \in \Sper R$ and let $h$ be a
  nonsingular hermitian form over $(A,\s)$. Then $\sign^\eta_\alpha h =
  \sign^\eta_\beta h$.
\end{prop}

\begin{proof}
  This follows from Lemma~\ref{cont-spear} and Theorem~\ref{sign-cont}.
\end{proof}

\begin{cor}
  Let $\eta$ be a reference form for $(A,\s)$ and assume that $\eta$ is
  nonsingular.
  Let $\alpha\subseteq \beta \in \Sper R$. Then $m_\alpha(A,\s) = 
  m_\beta(A,\s)$.    
\end{cor}

\begin{proof}
  Let $a \in \Sym(A^\x,\s)$. By Proposition~\ref{cst-spear}, we have
  $\sign^\eta_\alpha \qf{a}_\s = \sign^\eta_\beta \qf{a}_\s$. The result
  follows by taking the supremum over all such elements $a$.
\end{proof}

\begin{rem}
For central simple algebras with involution over a field, we established the 
existence of reference forms in \cite{A-U-Kneb} and \cite{A-U-prime}. 
By \cite[Proposition~A.3]{A-U-prime} we may assume that these are 
nonsingular.
We currently do not know if 
the same behaviour can be expected for Azumaya algebras with involution 
$(A,\s)$ without hypotheses on $(A,\s)$ or $R$.
\end{rem}


%

\begin{thebibliography}{10}

\bibitem{A-U-Kneb}
Vincent Astier and Thomas Unger.
\newblock Signatures of hermitian forms and the {K}nebusch trace formula.
\newblock {\em Math. Ann.}, 358(3-4):925--947, 2014.

\bibitem{A-U-prime}
Vincent Astier and Thomas Unger.
\newblock Signatures of hermitian forms and ``prime ideals'' of {W}itt groups.
\newblock {\em Adv. Math.}, 285:497--514, 2015.

\bibitem{A-U-PS}
Vincent Astier and Thomas Unger.
\newblock Signatures of hermitian forms, positivity, and an answer to a
  question of {P}rocesi and {S}chacher.
\newblock {\em J. Algebra}, 508:339--363, 2018.

\bibitem{A-U-pos}
Vincent Astier and Thomas Unger.
\newblock Positive cones on algebras with involution.
\newblock {\em Adv. Math.}, 361:106954, 48, 2020.

\bibitem{A-U-pc-gauge}
Vincent Astier and Thomas Unger.
\newblock Positive cones and gauges on algebras with involution.
\newblock {\em Int. Math. Res. Not. IMRN}, (10):7259--7303, 2022.

\bibitem{A-U-Az-PLG}
Vincent Astier and Thomas Unger.
\newblock Pfister's local-global principle for {A}zumaya algebras with
  involution.
\newblock {\em Doc. Math.}, 30(6):1325--1364, 2025.

\bibitem{A-U-sign-pc}
Vincent Astier and Thomas Unger.
\newblock Signature maps from positive cones on algebras with involution.
\newblock {\em To appear in Canadian Journal of Mathematics. Preprint at
  https://arxiv.org/abs/2505.22178}, 2025.

\bibitem{bourbaki-AC-1-4}
Nicolas Bourbaki.
\newblock {\em \'{E}l\'{e}ments de math\'{e}matique}.
\newblock Masson, Paris, 1985.
\newblock Alg\`ebre commutative. Chapitres 1 \`a 4. [Commutative algebra.
  Chapters 1--4], Reprint.

\bibitem{first23}
Uriya~A. First.
\newblock An 8-periodic exact sequence of {W}itt groups of {A}zumaya algebras
  with involution.
\newblock {\em Manuscripta Math.}, 170(1-2):313--407, 2023.

\bibitem{garrel-2023}
Nicolas Garrel.
\newblock Mixed {W}itt rings of algebras with involution.
\newblock {\em Canad. J. Math.}, 75(2):608--644, 2023.

\bibitem{KSU}
Manfred Knebusch and Claus Scheiderer.
\newblock {\em Real algebra---a first course}.
\newblock Universitext. Springer, Cham, 2022.
\newblock Translated from the 1989 German edition and with contributions by
  Thomas Unger.

\bibitem{knus91}
Max-Albert Knus.
\newblock {\em Quadratic and {H}ermitian forms over rings}, volume 294 of {\em
  Grundlehren der mathematischen Wissenschaften}.
\newblock Springer-Verlag, Berlin, 1991.
\newblock With a foreword by I. Bertuccioni.

\bibitem{BOI}
Max-Albert Knus, Alexander Merkurjev, Markus Rost, and Jean-Pierre Tignol.
\newblock {\em The book of involutions}, volume~44 of {\em American
  Mathematical Society Colloquium Publications}.
\newblock American Mathematical Society, Providence, RI, 1998.
\newblock With a preface in French by J. Tits.

\bibitem{KO}
Max-Albert Knus and Manuel Ojanguren.
\newblock {\em Th\'{e}orie de la descente et alg\`ebres d'{A}zumaya}.
\newblock Lecture Notes in Mathematics, Vol. 389. Springer-Verlag, Berlin-New
  York, 1974.

\bibitem{LQF2}
T.~Y. Lam.
\newblock {\em Introduction to quadratic forms over fields}, volume~67 of {\em
  Graduate Studies in Mathematics}.
\newblock American Mathematical Society, Providence, RI, 2005.

\bibitem{Lewis94}
D.~W. Lewis.
\newblock Trace forms of central simple algebras.
\newblock {\em Math. Z.}, 215(3):367--375, 1994.

\bibitem{L-U}
David~W. Lewis and Thomas Unger.
\newblock A local-global principle for algebras with involution and {H}ermitian
  forms.
\newblock {\em Math. Z.}, 244(3):469--477, 2003.

\bibitem{mahe82}
Louis Mah\'{e}.
\newblock Signatures et composantes connexes.
\newblock {\em Math. Ann.}, 260(2):191--210, 1982.

\bibitem{Reiner2003}
I.~Reiner.
\newblock {\em Maximal orders}, volume~28 of {\em London Mathematical Society
  Monographs. New Series}.
\newblock The Clarendon Press, Oxford University Press, Oxford, 2003.
\newblock Corrected reprint of the 1975 original, With a foreword by M. J.
  Taylor.

\bibitem{Saltman99}
David~J. Saltman.
\newblock {\em Lectures on division algebras}, volume~94 of {\em CBMS Regional
  Conference Series in Mathematics}.
\newblock Published by American Mathematical Society, Providence, RI; on behalf
  of Conference Board of the Mathematical Sciences, Washington, DC, 1999.

\bibitem{S24}
Claus Scheiderer.
\newblock {\em A course in real algebraic geometry. {Positivity} and sums of
  squares}, volume 303 of {\em Grad. Texts Math.}
\newblock Cham: Springer, 2024.

\bibitem{Zhang}
Fuzhen Zhang.
\newblock Quaternions and matrices of quaternions.
\newblock {\em Linear Algebra Appl.}, 251:21--57, 1997.

\end{thebibliography}

\end{document}